\documentclass[oneside,english]{amsart}
\usepackage[T1]{fontenc}
\usepackage[utf8]{inputenc}
\usepackage{amstext}
\usepackage{amsthm}
\usepackage{amssymb}

\makeatletter
\numberwithin{equation}{section}
\numberwithin{figure}{section}
\theoremstyle{plain}
\newtheorem{thm}{\protect\theoremname}
\theoremstyle{plain}
\newtheorem{lem}[thm]{\protect\lemmaname}
\newtheorem{rem}[thm]{\protect\remarkname}
\makeatother

\usepackage{babel}
\providecommand{\lemmaname}{Lemma}
\providecommand{\theoremname}{Theorem}
\providecommand{\remarkname}{Remark}

\begin{document}
\global\long\def\tauo{\tau_{\mathrm{ocean}}}%

\global\long\def\cauo{c_{\mathrm{ocean}}}%

\global\long\def\dev{\mathrm{Dev}\,}%

\global\long\def\sym{\mathrm{sym}\,}%

\global\long\def\D{D(u)}%

\global\long\def\div{\mathrm{div}\,}%

\global\long\def\tr{\mathrm{tr}\,}%

\global\long\def\del{\delta}%

\global\long\def\R{\mathbb{R}}%

\global\long\def\deld{\del_{\bullet}}%

\global\long\def\delu{\del^{\bullet}}%

\global\long\def\eps{\varepsilon}%

\global\long\def\Uo{U_{\mathrm{ocean}}}%

\global\long\def\cb{c_{\bullet}}%

\global\long\def\Cb{C_{\bullet}}%

\global\long\def\cbu{c^{\bullet}}%

\global\long\def\Cbu{C^{\bullet}}%

\global\long\def\mF{\mathcal{F}}%

\global\long\def\ut{\tilde{u}}%

\global\long\def\tA{\tilde{A}}%

\title[Global existence for Hibler's model]{Global existence and uniqueness for Hibler's visco-plastic sea-ice
model}
\author{Stefan Dingel and Karoline Disser}
\address{Universität Kassel\\
 Institut für Mathematik \\
 Heinrich-Plett-Stra{ß}e 40 \\
 34132 Kassel, Germany }
\email{sdingel@mathematik.uni-kassel.de}
\email{kdisser@mathematik.uni-kassel.de}
\date{August 22, 2025}
\begin{abstract}
In this paper, we prove global existence and uniqueness of weak solutions to the momentum equations of Hibler's visco-plastic model for the dynamics of the arctic sea-ice covers. Although Hibler's model is standardly used in global climate simulations, there are only few rigorous mathematical results so far that mainly concern local-in-time well-posedness of globally regularized variants. Here, we consider Hibler's original model with local cut-off for arbitrarily small and large strain rates. Degeneracy and plasticity of the stress tensor hold in this range.   \end{abstract}

\thanks{The authors gratefully acknowledge the financial support by DFG project FOR 5528.}
\maketitle

\section{Introduction}

\subsection{Hibler's sea-ice model}

The arctic sea-ice covers strongly impact global climate. Modelling and simulation of their dynamics is a very challenging topic due to their complex material and thermodynamical behaviour, strong coupling to oceanic and atmospheric conditions and the difficulty of obtaining relevant data, \cite{BCEFGHHJPPSSW20}, \cite{HABB2020}. For modeling and computational reasons, Hibler's visco-plastic sea-ice rheology \cite{Hib79} is used as a standard in global climate models. At the same time, rigorous mathematical analysis of Hibler's model is still largely missing. \\
The topic has recently gained increasing attention \cite{BDDH22,LTT22, BH23, BBH24}, but existence of global solutions remains to be an open problem, and  results so far concern local well-posedness and regularized versions of Hibler's model. We also refer to \cite{BLTT25}, where global existence of the related elasto-visco-plastic (EVP) model was proved, assuming a Kelvin-Voigt regularization. \\
Here, we prove the first result on global existence of solutions for Hibler's sea-ice momentum balance, given by  
\begin{equation}
\mathrm{m}(\dot{u}+u\cdot\nabla u)-\mathrm{div}(\sigma)=\mathrm{m}\omega u^\perp+\tau_{\mathrm{ocean}}(u)+\tau_{\mathrm{atm}}-\mathrm{m}\mathrm{g}\nabla H+f,\label{eq:eq}
\end{equation}
where, up to any time $T>0$, $u\colon[0,T]\times\Omega\to\mathbb{R}^{2}$
denotes the unknown velocity of the sea-ice cover on a bounded Lipschitz
domain $\Omega\subseteq\mathbb{R}^{2}$, and mass $\mathrm{m}>0$,
Coriolis parameter $\omega\in\mathbb{R}$, and gravity $g>0$
are given constants, and $\tauo,\tau_{\mathrm{atm}},H$ and $f$ are
forces due to external conditions (details below). Here, $v^{\perp}=(-v_{2},v_{1})^{\top}$ denotes the vector perpendicular to $v=(v_{1},v_{2})^{\top}\in\mathbb{R}^{2}$. 
On the boundary $\partial \Omega$, we assume homogenous Dirichlet
boundary conditions. 
In Hibler's model, the visco-plastic stress tensor
\[
\sigma=\sigma(Du)=\sigma(P,Du)=\frac{P}{2}\left[\frac{\lambda\left(\dev Du\right)+\left(\div u\right)\mathrm{Id}}{\delta(Du)}-\mathrm{Id}\right]
\]
is derived from the (plastic) tensor
\[
\sigma_{p}(P,Du)=\begin{cases}
    \frac{P}{2}\left[\frac{\lambda\left(\dev Du\right)+\left(\tr Du\right)\mathrm{Id}}{\sqrt{\lambda|\dev Du|^{2}+(\tr Du)^{2}}}-\mathrm{Id}\right],&\text{if } Du\neq0,\\
    0,&\text{if } Du=0,
\end{cases}
\]
by applying a (local) cut-off to the strain rate in the denominator, see (\ref{eq:defdel}) below.
Here,
\[
Du:=\sym \nabla u := \frac{1}{2}\left(\nabla u+(\nabla u)^\top\right)
\]
denotes the symmetric part of the gradient of $u$, and the constant $\lambda=\frac{2}{\bar{e}^{2}}< 1$ is derived from the
parameter $\bar{e} \approx 2 $ in Hibler's model that describes the ratio
of ellipticity of the yield curve. For the analysis, it is only required
that $\bar{e}>0$. The function $P\colon[0,T]\times\Omega\to[P_{\bullet},\infty)$
is a given parameter of local ice strength. In the following, for given
matrices $z,\tilde{z}\in\mathbb{R}^{2\times2}$ with coefficients
$z_{ij},\tilde{z}_{ij},i,j\in\{1,2\}$, we use the short-cuts
\[
\tr z=z_{11}+z_{22},\qquad\dev z=\frac{1}{2}\left(\begin{array}{cc}
z_{11}-z_{22} & z_{12}+z_{21}\\
z_{12}+z_{21} & z_{22}-z_{11}
\end{array}\right),
\]
and the Frobenius norm $\vert z\vert$ induced by the scalar product
\[
z:\tilde{z}:=\sum_{i,j}z_{ij}\tilde{z}_{ij}.
\]
The \emph{visco-plasticity}
of the stress tensor $\sigma$ is due to the inverse viscosity function
\begin{equation}
\delta\colon\R^{2\times2}\ni z\mapsto\begin{cases}
\del^{\bullet}, & \delta_{p}(z)\geq\del^{\bullet},\\
\del_{p}(z):=\sqrt{\lambda|\dev z|^{2}+(\tr z)^{2}}, & \del_{\bullet}\leq\del_{p}(z)\leq\del^{\bullet},\\
\del_{\bullet}, & \del_{\bullet}\geq\delta_{p}(z),
\end{cases}\label{eq:defdel}
\end{equation}
with cut-off at an arbitrarily large fixed value $\delta^{\bullet}>0$
and at an arbitrarily small fixed value $\del_{\bullet}>0$. Examples
of explicit values for $\deld,\delu$ are given in \cite{Hib79}.

\subsubsection*{Properties of $\sigma$ and $\sigma_{p}$}

The (degenerate) plastic stress tensor $\sigma_{p}$ has several distinct
properties that are difficult to deal with in numerical simulation
and analysis:
\begin{enumerate}
\item \label{deg}\emph{degeneracy:} if $\dev z=0\in\mathbb{R}^{2\times2}$, then a
direct calculation shows that $\sigma_{p}(z)=0$, i.e the material
becomes (locally) stress-free. This was intended by Hibler due to
the high resistance of sea ice to compression.
\item \label{plast}\emph{plasticity:} the equality
\begin{equation}
\frac{1}{4}|\tr(\sigma_{p}(z)+\frac{P}{2}\mathrm{Id})|^{2}+\frac{1}{\lambda}|\dev(\sigma_{p}(z)+\frac{P}{2}\mathrm{Id})|^{2}=\frac{P^{2}}{4},\quad\text{for all }z\in\R^{2\times2}, \label{eq:yieldcurve}
\end{equation}
shows that the principle stresses (the eigenvalues of $\sigma_{p}$)
are normed to an elliptic yield curve. This independence of stresses
of the strain rate is desirable in general, particularly for the small
rates $\delta_{p}$ that are expected here, but it can also be viewed
as an arbitrarily large viscosity occuring as $\del_{p}\to0$. This
is the main reason why Hibler suggests a cut-off in $\del$. He also suggests to use a value
$\del_{\bullet}$ well below what will appear in simulations.
\item \label{disc} \emph{discontinuity: }clearly, $\sigma_{p}$ is discontinuous at $z=0$.
\item \label{coerc} \emph{non-coercivity: }this is a mathematical issue associated to
(\ref{eq:yieldcurve}): it is not true that $|\sigma_{p}(z)|\to\infty$
as $z\to\infty$ (weak coercivity), so there are no straightforward 
a-priori estimates, e.g. on an Galerkin or other type of approximation. 
\end{enumerate}
Hibler's stress tensor $\sigma$, which uses the modified strain rate
$\del$ in (\ref{eq:defdel}), has the following properties:
\begin{enumerate}
\item degeneracy is still present because $\sigma(z)=0$ remains true as long
as $\dev z=0$ and $\del_{\bullet}\leq|\tr z|\leq\delta^{\bullet}$.
\item Equality (\ref{eq:yieldcurve}) also remains true for the typical values
$\del_{\bullet}\leq\del_{p}(z)\leq\delta^{\bullet}$. Otherwise, principle
stresses remain constrained within the ellipse ($\del_{p}(z)\leq\del_{\bullet}$)
or increase linearly in $|z|$ ($\del^{\bullet}\leq\del_{p}(z)$).
\item The cut-off for small $\del_{p}(z)$ turns $\sigma$ into a continuous
function. It is however not differentiable at $\deld,\delu$ and thus
often referred to as \emph{non-smooth}.
\item Due to the cut-off at large values of $\del_{p}$, the operator $\div\sigma(D \cdot ) \colon H_{0}^{1}\to H^{-1}$
is coercive in a suitable sense, see Lemma \ref{lem:coercive}.
\end{enumerate}
To summarize, one can say that Hibler's stress tensor with local cut-off
respects the two key material properties (\ref{deg}) and (\ref{plast}) in an arbitrarily
large range of parameters, but helps to solve mathematical issues
associated to the plastic tensor $\sigma_{p}$ by improving on (\ref{disc})
and (\ref{coerc}). This is an advantage with respect to other, \emph{global} regularizations
or modifications used in the literature. For example, the proof of
local well-posedness in \cite{BDDH22} and \cite{LTT22} relies on the $\varepsilon$-regularization
of $\delta_{p}$ to 
\begin{equation} \label{deltaeps}
	\delta_{p}^{\varepsilon}(z):=\sqrt{\eps+\lambda|\dev z|^{2}+(\tr z)^{2}}.
	\end{equation}
Although $\delta_{p}^{\eps}$ gives a better fit to $\del_{p}$ than
$\del$ at large $z$, it destroys properties
(\ref{deg}) and (\ref{plast}) and induces a change of type in turning the momentum
equations into a quasilinear parabolic system, \cite{BDDH22}.
In numerical analysis, typically both $\eps$-regularization
and cut-off are used \cite{MK2021}. Our main result here extends to this case, see Subsection \ref{ssec:eps}.
The proof of global existence for the related
EVP model in \cite{BLTT25} uses the global Kelvin-Voigt $\alpha$-regularization in
a dynamic formulation for $\sigma$. Using $\alpha>0$,
it is shown that it is possible to pass to the limit $\eps\to0$
in the additional $\eps$-regularization. However,
$\alpha>0$ changes the type of the original elasto-visco-plastic
formulation (which seems to be ill-posed if $\alpha=0$).

Properties (\ref{plast}) - (\ref{coerc}) of $\sigma_{p}$ are shared
by the total variation flow associated to the 1-Laplacian \cite{ABCM2001}.
It is not yet well-studied as a system that uses the symmetric gradient
only, or in a degenerate setting, or including additional non-linearities, see \cite{BCE2025} and references therein for recent results. 
Hibler's modification by local cut-off provides a particular form
of approximation that preserves (\ref{deg}), (\ref{plast}), locally, unlike global approximations by $\varepsilon >0$, \cite{FP2003}, $p>1$, \cite{GK2019}, or added viscosity, used for the related model of Bingham fluids \cite{DL76}. Investigations of the singular limits $\del_{\bullet}\to0,\del^{\bullet}\to\infty$
may be the subject of future work.

\subsubsection*{External forces}

In (\ref{eq:eq}),
\begin{equation}
\tauo=\tauo(u)=c_{\mathrm{ocean}}|\Uo-u|(\Uo-u)_{\theta}\label{eq:tauo}
\end{equation}
refers to a typical forcing of the ice sheet due to the given geostrophic
ocean current $\Uo$, where $c_{\mathrm{ocean}}\geq0$ is a given
constant and $\theta$ denotes the water turning angle. Here, for
all $v\in\R^{2}$, $v_{\theta}=\cos\theta v+\sin\theta v^{\perp}$
denotes the vector turned by angle $\theta$.
For the analysis, we assume 
$0\leq\theta\leq\frac{\pi}{4}$ and note that this is a standard assumption
in the literature, \cite{HTCF2019}. For simplicity of notation, we
write 
$$ h:=\tau_{\mathrm{atm}}-\mathrm{m}\mathrm{g}\nabla H+f, $$ 
where
$\tau_{\mathrm{atm}}$ denotes forcing due to air flow above the ice
sheet but is typically assumed to be indenpendent of $u$, as is the
force $-\mathrm{m}\mathrm{g}\nabla H$ due to varying height $H$
of the sea surface, cf.\ \cite{Hib79}. The term $f$ can be used to collect additional external forces. For typical values of physical parameters in this model, see \cite{Hib79} or e.g.\ Table 1 in \cite{MDLLHRBHK2021}. 

\subsubsection*{Convection term}

The term $\mathrm{m}(u\cdot\nabla u)$ is typically omitted in simulations and analysis
due to its smallness \cite{BLTT25,MDLLHRBHK2021}. Viewed as a Lipschitzian semilinearity, this term can be included in our analysis by a fixed-point-argument, either locally in time or for small initial data. However, its global estimates would depend on a coherent modelling of the transport of mass (as in compressible
flows), or on incompressibility. In future work, we will investigate the coupled thermodynamic aspects of sea ice dynamics associated to Hibler's model, with varying coupled ice strength $P$ and mass $\mathrm{m}$, and include convection in this context. 

\subsection{Main result}

Let $V:=H_{0}^{1}(\Omega;\R^{2})$ denote the usual first-order Hilbert
Sobolev space with zero boundary trace. Then 
$$ V\overset{c}{\hookrightarrow}H := L^{2}(\Omega;\R^{2})\simeq H^{*}\overset{c}{\hookrightarrow}V^{*}, $$
with $V^{*}=H^{-1}(\Omega;\R^{2})$ the dual space of $V$, forms a
Gelfand triple with compact dense embeddings. The main result is the
following.
\begin{thm}
\label{thm:mr}Let $T>0$ and $u_{0}\in H$ be given. Assume that
\begin{align*}
\Uo & \in L^{\infty}(0,T;H)\cap L^{2}(0,T;L^{4}), \\
 h & \in L^{2}(0,T;V^{*}), \\
P & \in L^{\infty}(0,T;L^{\infty}(\Omega;[P_{\bullet},\infty)), \quad \text{for some } P_\bullet >0. 
\end{align*} 
Then there exists a unique weak solution
\[
u\in  C([0,T];H)\cap L^{2}(0,T;V), \quad \dot{u} \in L^{2}(0,T;V^{*})
\]
of (\ref{eq:eq}) with initial value $u(0)=u_{0}$.
\end{thm}

\subsection{Literature and discussion}

For a general overview of results and challenges in mathematical sea-ice modeling, we refer to the survey \cite{BCEFGHHJPPSSW20} and references therein. Hibler's (VP) model \cite{Hib79} was originally devised based on the results of the AIDJEX program, \cite{CMPRT1974}. For the study of numerical schemes, based on this model, and of its stability, performance and comparison to related models we refer e.g.\ to \cite{ZH97,  Gray99, KHLF00, LMCHH10,  GLS13, SK18, MDLLHRBHK2021} and references therein.  \\
The two papers \cite{LTT22} and \cite{BDDH22} independently established local well-posedness for Hibler's model with slightly varying $\varepsilon$-regularizations, see also \cite{Brandt24}. In \cite{BDDH22}, global existence is proved for small initial data. These results were partially extended to the cases of interaction with a rigid body \cite{BBH24}, with ocean and atmosphere \cite{BBH24Arxiv}, and to the time-periodic setting \cite{BH23}. Numerically relevant a-priori estimates for the $\varepsilon$-regularized model with cut-off were also derived in \cite{MK2021}. A linear setting with different regularization was studied in \cite{CK23}. Related results on the ill- and well-posedness as well as change of type in Hibler's thermodynamic model without regularization were obtained in \cite{Gray99, GLS13}. Note that in this paper, we obtain well-posedness in the finite kinetic energy setting, due to estimate (\ref{wellposed}).   \\
Global existence of solutions for large data was shown recently for the (\emph{elasto}-visco-plastic) EVP-model with Kelvin-Voigt regularization \cite{BLTT25}, which was derived from Hibler's model to achieve numerical stabilization \cite{Hunke1997,DH02}.  Implementations of Hibler's model and the EVP model are both established in climate models. The simulated behaviour may differ substantially, and their respective efficiency is up to debate \cite{KDL2015, LKTHL12, LMCHH10}.  Note that the EVP-model analysed in \cite{BLTT25} includes a dynamical formulation for the stress tensor that needs additional initial data on these quantitities that are hard to come by, whereas Hibler's model was chosen to avoid this difficulty \cite{Hib79}. In comparison to the analysis in a strong setting in \cite{BLTT25}, here, we are working in a weak functional analytic framework correpsonding to finite kinetic energy at initial time and our analysis on bounded domains allows us to consider general (also non-periodic) external forces.

\section{Proof of Theorem \ref{thm:mr}}

\subsection{Hibler's operator}

We are looking for a solution $u$ to the weak momentum balance
\[
\langle\mathrm{m}\dot{u}(t),v\rangle+a(t,u(t),v)+g(t,u(t),v)+c(u(t),v)=\langle h(t),v\rangle,
\]
for almost all $t\in(0,T)$ and all $v\in V$, where $\langle\cdot,\cdot\rangle$
denotes the dual pairing for $V^{*}\times V$, 
\begin{align*}
a(t,u,v) & :=\int_{\Omega}\sigma(P(t),Du):Dv\\
 & =\int_{\Omega}\frac{P(t)}{2}\left[\frac{\lambda\dev Du:\dev Dv+\div u\,\div v}{\delta(Du)}-\div v\right]
\end{align*}
denotes the form $a\colon[0,T]\times V\times V\to\R$ associated to
Hiber's stress tensor, and
\begin{equation}
g(t,u,v):=-\int_{\Omega}\tauo(u)\cdot v,\qquad c(u,v):=-\int_{\Omega}\mathrm{m}\omega u^\perp\cdot v,\label{eq:defgc}
\end{equation}
are the forms corresponding to the external forces that depend on
$u$.
We consider Hibler's operator given by
\[
A(t)\colon V\to V^{*},\qquad\langle A(t)u,v\rangle:=a(t,u,v),
\]
in the framework of monotone operators. As a reference, we also define
\begin{align*}
a_{p}(t,u,v) & :=\int_{\Omega}\sigma_{p}(P(t),Du):Dv, 
\end{align*}
as the form corresponding to the plastic stress tensor $\sigma_{p}$
without cut-off  and the corresponding operator
\[
A_{p}(t) \colon V\to V^{*},\qquad\langle A_{p}(t)u,v\rangle:=a_{p}(t,u,v).
\]

\begin{lem}
\label{lem:monotone}The operators $A(t)$ are monotone
and hemicontinuous, i.e. for all $t\in(0,T)$,
\begin{equation}
\langle A(t)u-A(t)v,u-v\rangle\geq0,\qquad\text{for all }u,v\in V,\label{eq:monotonicity}
\end{equation}
and for all $u,v,w\in V$, the map
\begin{equation}
[0,1]\ni h\mapsto a(t,u+hv,w)\label{eq:hemicont}
\end{equation}
is continuous. Furthermore, the operators $A_{p}(t)$ are monotone.
\end{lem}

\begin{proof}
To shorten the notation, for all $u\in V$, set
\[
D^{\lambda}u:=\lambda \dev Du + (\div u) \mathrm{Id}.
\]
Let $u,v\in V$ be given. By definition and using Cauchy-Schwarz,
\begin{equation}
    D^{\lambda}u:Du=\del_{p}^{2}(Du),\qquad\text{and }D^{\lambda}u:Dv\leq\delta_{p}(Du)\delta_{p}(Dv).\label{eq:dudv}
\end{equation}
By assumption, $P(t)\geq 0$. Using (\ref{eq:dudv}), the calculations are straightforward,
\begin{align*}
&\phantom{=}\langle A(t)u-A(t)v,u-v\rangle\\&= \int_{\Omega}\frac{P(t)}{2}\left(\frac{D^{\lambda}u:Du}{\delta(Du)}-\frac{D^{\lambda}u:Dv}{\delta(Du)}-\frac{D^{\lambda}v:Du}{\delta(Dv)}+\frac{D^{\lambda}v:Dv}{\delta(Dv)}\right)\\ 
&\geq \int_{\Omega}\frac{P(t)}{2}\left(\del_{p}(Du)-\del_{p}(Dv)\right)\left(f(\del_{p}(Du)) -f(\del_{p}(Dv)) \right),
\end{align*}
where $f(\del_{p}(Du)) := \frac{\del_{p}(Du)}{\del(Du)}$ is an increasing function in terms of $\del_{p}(Du)$, so that the monotonicity of $A(t)$ directly follows from the monotonicity of $f$. Monotonicity of $A_{p}(t)$ follows in the same way, where $f$ would be replaced by $f_p \equiv 1$. 

For the proof of (\ref{eq:hemicont}), note that as $u,v,w$ are chosen
arbitrarily in the vector space $V$, it is sufficient to check continuity
in one point, for example, in $h=0$. Given $h_{n}\overset{n\to\infty}{\longrightarrow}0$,
consider
\begin{align}
 & |a(t,u+h_{n}v,w)-a(t,u,w)|\label{eq:acont}\\
 & \leq\int_{\Omega}\frac{P(t)}{2}\left[|D^{\lambda}u:Dw|\left|\frac{1}{\del(Du+h_{n}Dv)}-\frac{1}{\del(Du)}\right|+h_{n}\frac{\left|D^{\lambda}v:Dw\right|}{\del(Du+h_{n}Dv)}\right]\nonumber
\end{align}
The continuity of $z\mapsto\del(z)$ and the strict positivity, $\delta(z)\geq\deld>0$,
imply that the integrands converge pointwise a.e. We can use
\begin{equation}
\lambda |\sym z|^2=\lambda |\dev z|^{2}+\frac{\lambda}{2}(\tr z)^{2}\leq \lambda |\dev z|^{2}+\frac{1}{2}(\tr z)^{2} =  \delta_{p}^2(z)\leq|\sym z|^2,\label{eq:betragvsdel}
\end{equation}
(here, $\lambda\leq1$, otherwise, only the constants would be different),
and both Korn's and Poincaré's inequalities to obtain equivalence
of norms,
\begin{equation}
\Vert u\Vert_{V}^{2}\lesssim\Vert Du\Vert_{2}^{2}\lesssim\int_{\Omega}\delta_{p}^{2}(Du)\lesssim\Vert Du\Vert_{2}^{2}\lesssim\Vert u\Vert_{V}^{2}.\label{eq:kornpoinc}
\end{equation}
Thus, convergence to $0$ of the integral in (\ref{eq:acont}) follows
from dominated convergence using (\ref{eq:dudv}) and again the strict
positivity of $\del$. 
\end{proof}

\begin{rem}
    Clearly, $A(t)$ is not strictly monotone. As an example, for every  $u \in V$ such that $\del(Du) \in (\del_{\bullet},\del^{\bullet})$ almost everywhere,
    there is an $\varepsilon > 0$ such that for $v := (1+\varepsilon) u$, we also have $\del(Dv) \in (\del_{\bullet},\del^{\bullet})$. In this case,  $\langle A(t)u-A(t)v,u-v\rangle=0$.
\end{rem}

\begin{lem}
\label{lem:coercive}Uniformly in $t\in(0,T]$, the operators $A(t)$
are coercive with
\begin{equation}
\langle A(t)u,u\rangle\geq\cb\Vert u\Vert_{V}^{2}-\Cb,\qquad\text{for all }u\in V, \label{eq:Acoerc-1}
\end{equation}
and satisfy the growth bound
\begin{equation}
\Vert A(t)u\Vert_{V^{*}}\leq C^\bullet + \cbu \Vert u\Vert_{V},\qquad\text{for all }u\in V.\label{eq:Abdd-1}
\end{equation}
Moreover, the function
\[
(0,T)\ni t\mapsto\langle A(t)u,v\rangle
\]
is measurable for all $u,v\in V$.
\end{lem}

\begin{proof}
For all $u,v\in V$, using (\ref{eq:dudv}), (\ref{eq:betragvsdel}) and (\ref{eq:kornpoinc}),
we obtain
\begin{equation*}
    a(t,u,u)  \geq \frac{1}{2\delu}\int_\Omega P(t)\left(\delta_{p}^{2}(Du)-\delu|\div u|\right) \geq \frac{c(\Omega)P_\bullet\lambda}{2\delu} \Vert u \Vert_V^2 - \frac{{\delu}^2\Vert P(t) \Vert_\infty}{2(1-\lambda)},
\end{equation*}
 and
\begin{equation*}
|a(t,u,v)|\leq C(\Omega,\lambda)\Vert P(t)\Vert_{\infty}\left(1 +\frac1{\deld} \Vert u\Vert_{V} \right)\Vert v\Vert_{V}.
\end{equation*}
The (weak) measurability of $A(\cdot)$ is implied by $P\in L^{1}(0,T;L^{\infty}(\Omega))$. 
\end{proof}

\subsection{Coriolis term}

Due to
\[
c(u,u-v)-c(v,u-v)=0,\quad c(u,u)=0,
\]
and
\[
c(u,v)\leq \mathrm m \vert \mathrm{\omega} \vert C(\Omega)\Vert u\Vert_{V}\Vert v\Vert_{V},
\]
the operators
\[
A(t)+C\colon V\to V^{*},\,\,\langle\left(A(t)+C\right)u,v\rangle=a(t,u,v)+c(u,v),
\]
are also monotone, hemicontinuous, coercive, and satisfy the growth condition \eqref{eq:Abdd-1}. 

\begin{rem}
    The properties above also hold for non-constant $\mathrm \omega \in L^\infty((0,T);L^r(\Omega))$, with $r>1$.  
\end{rem}

\subsection{Oceanic forces}

It remains to study the time-dependent family of operators
\[
t\mapsto G(t)\colon V\to V^{*},\qquad\langle G(t)u,v\rangle:=-\cauo\int_{\Omega}|\Uo(t)-u|(\Uo(t)-u)_{\theta}\cdot v,
\]
associated to the oceanic currents that force the system.

\begin{lem}
    \label{lem:Gbounds}The family of operators $G(t)\colon V\to V^{*}$
    is well-defined and measurable w.r.t. time. Furthermore,  for every $\varepsilon > 0$ there are functions $U_{\bullet}\in L^{1}(0,T;[0,\infty))$, $U^{\bullet}\in L^{2}(0,T;[0,\infty))$ and a constant $c^G > 0$, such that the operators $G(t)$ are hemicontinous and satisfy
    \begin{equation}
        \langle G(t)u,u\rangle\geq-\varepsilon\Vert u\Vert_{V}^{2}-U_{\bullet}(t),\qquad\text{for all }u\in V, \label{eq:Gcoerc}
    \end{equation}
    as well as the growth bound
    \begin{equation}
        \Vert G(t)u\Vert_{V^{*}}\leq U^{\bullet}(t)+c^G\Vert u\Vert_{V}\Vert u\Vert_{H},\qquad\text{for all }u\in V.\label{eq:Gbdd}
    \end{equation}
\end{lem}

\begin{proof}
    The family $t\mapsto G(t)$ is clearly well-defined and (weakly) measurable. Hemicontinuity and 
   (\ref{eq:Gbdd}) follow from 
    the estimate
    \begin{multline*}
        \int_{\Omega}|\Uo(t)-u|(\Uo(t)-u)_{\theta}\cdot v\label{eq:taudefined-1}\\
         \leq2(\Vert\Uo(t)\Vert_{H}\Vert\Uo(t)\Vert_{L^{4}(\Omega)}+\Vert u\Vert_{H}\Vert u\Vert_{L^{4}(\Omega)})\Vert v\Vert_{L^{4}(\Omega)}.
    \end{multline*}
    Using that $\cos\theta>0$, the bound in (\ref{eq:Gcoerc}) follows from
    \begin{align*}
        \langle G(t)u,u\rangle & =-\int_{\Omega}|\Uo(t)-u|(\Uo)_{\theta}\cdot u+\int_{\Omega}|\Uo(t)-u|(\cos\theta)|u|^{2}\\
        & \geq-\frac{1}{4\cos\theta}\int_{\Omega}|\Uo(t)-u||\Uo(t)|^{2},
	  \end{align*}	
	  and the integrability assumptions on $\Uo$. 
\end{proof}

\begin{lem}
    \label{lem:Gmonotone}If $0\leq\theta\leq\frac{\pi}{4}$, then the operators
    $G(t)$ are monotone.
\end{lem}

\begin{proof}
   Pointwise, we prove the non-negativity of the integrand in 
    \begin{equation*}
        \langle G(t)u-G(t)v,u-v\rangle = \int_{\Omega}\begin{aligned}[t]&\cos\theta\left(|\Uo(t)-u|^{3}+|\Uo(t)-u|^{3}\right)\\
         &-|\Uo(t)-u|\left(\Uo(t)-u\right)_{\theta}\left(\Uo(t)-v\right)\\
         &-|\Uo(t)-v|\left(\Uo(t)-v\right)_{\theta}\left(\Uo(t)-u\right).\end{aligned}
    \end{equation*}
    To simplify notation, for every $x\in\Omega$, set $a:=\Uo(t,x)-u(x)$,
    $b:=\Uo(t,x)-v(x)$ and $\varphi\in[0,\pi]$ as the angle between
    $a$ and $b$. Furthermore, define $\alpha:=|a| = |a_\theta|$ and $\beta:=|b| = |b_\theta|$.
    Then
    \begin{align*}
        a_{\theta}\cdot b & =\cos\theta\cos\varphi\,\alpha\beta+\sin\theta\sin\varphi\,\alpha\beta,\\
        b_{\theta}\cdot a & =\cos\theta\cos\varphi\,\alpha\beta-\sin\theta\sin\varphi\,\alpha\beta,
    \end{align*}
    and thus, the integrand can be expressed as the polynomial
    \[
    P(\alpha,\beta):=\cos\theta\left(\alpha^{3}+\beta^{3}-\cos\varphi\alpha\beta(\alpha+\beta)-\tan\theta\sin\varphi\alpha\beta(\alpha-\beta)\right).
    \]
    We show that $P(\alpha,\beta)\geq0$ for $\alpha,\beta \geq 0$. This is clearly true if $\beta=0$.
    Otherwise, we can divide by $\frac{1}{\cos\theta\beta^{3}}$ and consider the
    rescaled polynomial
    \[
    p(\gamma):=\gamma^{3}-(\cos\varphi+\tan\theta\sin\varphi)\gamma^{2}+(-\cos\varphi+\tan\theta\sin\varphi)\gamma+1.
    \]
    The discriminant of $p$, written in terms of $T:=\tan^{2}\theta$
    and $S:=\cos\varphi$ with $\sin\varphi=\sqrt{1-S^{2}}$ is
    \begin{equation*}
        d(S,T):=(1-S^{2})^{2}T^{2}+(-2S^{2}+24S-18)(1-S^{2})T+(S^{4}+8S^{3}+18S^{2}-27).
    \end{equation*}
	Due to the condition $0\leq\theta\leq\frac{\pi}{4}$, we only need to
	    check $0\leq T\leq 1$ and $-1\leq S\leq 1$. If $S=1$, then $P(\alpha,\beta) = \cos \theta (\alpha-\beta)^2(\alpha + \beta) \geq 0$.
    At the $T$-boundaries of the relevant range, $d(S,0) < 0$, and
    \[
    d(S,1)=4S^{4}-16S^{3}+32S^{2}+24S-44< 0,
    \]
    so $d(S,T)<0$ for all $0< T <1$, because $d$ is quadratic in $T$
    and the leading coefficient is positive. This means that $p$ has
    only one real root. Due to $p(0)=1>0$ and $p(\lambda)\to\infty$
    as $\lambda\to\infty$, the real root must be negative. We have shown
    that $p>0$, which implies $P>0$ and thus, the monotonicity of $G(t)$.
\end{proof}

\begin{rem}
    Lemmas \ref{lem:Gbounds} and \ref{lem:Gmonotone} remain true for non-constant $$\theta \in L^\infty(0,T; L^\infty(\Omega; [0,\frac\pi4])).$$
\end{rem}

\subsection{A-priori estimates and stability}

Theorem \ref{thm:mr} is obtained by Galerkin approximation and using
monotonicity through Minty's trick. We refer e.g.\ to the monograph
of Zeidler \cite{Zei90B} 
and references therein. Here, we also need the result of Liu \cite{Liu11} on pseudomonotonicity for specific semilinear terms, in order to deal with the oceanic forcing. If $\tauo=0$, then by Lemmas
\ref{lem:monotone} and \ref{lem:coercive}, the hypothesis of Theorem
30.A in \cite{Zei90B} 
are verified and the result follows directly.
More precisely, it is shown that there
exists a unique solution
\[
u\in  C([0,T];H)\cap L^{2}(0,T;V), \dot{u} \in L^{2}(0,T;V^{*}),
\]
with initial value $u(0)=u_{0}$ of the equation
\begin{equation*}
    \dot{u}(t)+(A(t)+C)u(t)=h(t)\in V^{*},\label{eq:Atilde}
\end{equation*}
for a.a. $t\in(0,T)$, and $u$ satisfies
the corresponding a-priori estimate
\begin{equation*}
    \Vert u\Vert_{C([0,T];H)}+\frac{\cb}{2}\Vert u\Vert_{L^{2}(0,T;V)}^{2}\leq\Cb+\Vert u_{0}\Vert_{H}+c(\cb)\Vert h\Vert_{L^{2}(0,T;V^{*})}^{2}.\label{eq:apriori}
\end{equation*}
If $\tauo\neq0$, then by Lemmas \ref{lem:Gbounds} and \ref{lem:Gmonotone}, the corresponding family of operators
\[
t\mapsto A(t)+C+G(t):V\to V^*
\]
satisfies the hypotheses (H1)-(H4) in the paper by Liu, \cite{Liu11}, thus also proving Theorem
\ref{thm:mr}. In addition to uniqueness, by Theorem 1.2 in \cite{Liu11}, well-posedness holds in the sense that
\begin{equation}\label{wellposed}
\Vert u_1(t) - u_2(t) \vert_H^2 \leq \Vert u_{1,0} - u_{2,0} \Vert^2_H + \int_0^t \Vert h_1(s) - h_2(s) \Vert^2_H, 
\end{equation}
for all solutions $u_1,u_2$ corresponding to initial data $u_{1,0},u_{2,0} \in H$ and right-hand-sides 
\begin{equation}\label{hinH}
h_1, h_2 \in L^2(0,T;H),
\end{equation}
respectively. Note that the requirement (\ref{hinH}) is slightly stronger than the assumption $h \in L^2(0,T;V^*)$ in Theorem \ref{thm:mr}. 

\subsection{Hibler's operator with $\varepsilon$-regularization}\label{ssec:eps} 
A typical global regularization of Hibler's operator is given by replacing $\delta_{p}$ with $\delta^\varepsilon_p$ defined in (\ref{deltaeps}).  The corresponding operators 
$$ A_p^\varepsilon(t) \colon V\to V^{*},
\; \langle A^\varepsilon_{p}(t)u,v\rangle := \int_\Omega \frac{P(t)}2 \left( \frac{D^\lambda u : Dv}{\delta_{p}^\varepsilon(Du)} - \div v \right),
$$
are monotone, hemicontinuous and clearly also satisfy a growth bound of type (\ref{eq:Abdd-1}). The monotonicity follows as for $A(t)$, where in the proof, $f$ would be replaced by the monotone function $f_p^\varepsilon(x) := x/\sqrt{\varepsilon + x^2}$. The hemicontinuity also follows analogously  due to the strict positivity of $\delta_{p}^\varepsilon \geq \sqrt{\varepsilon} > 0$.
On the other hand, $A_p^\varepsilon(t)$ is not coercive. This can be changed by introducing a cut-off at large values, 
	$$  \delta_{p}^{\varepsilon,\bullet} (z) := \max\lbrace \delta_{p}^\varepsilon(z), \delu \rbrace, $$
and is also true in case of a cut-off on both sides, with $\delta_p$ in (\ref{eq:defdel}) replaced by $\delta^\varepsilon_p$, 
\begin{equation*}
\delta^\varepsilon \colon\R^{2\times2}\ni z\mapsto\begin{cases}
\del^{\bullet}, & \delta^\varepsilon_{p}(z)\geq\del^{\bullet},\\
\del^\varepsilon_{p}(z), & \del_{\bullet}\leq\del^\varepsilon_{p}(z)\leq\del^{\bullet},\\
\del_{\bullet}, & \del_{\bullet}\geq\delta^\varepsilon_{p}(z),
\end{cases}
\end{equation*}
which is standard in numerical simulations \cite{MK2021}. So in both cases, using either $\delta_{p}^{\varepsilon,\bullet}$ or $\delta^\varepsilon$ instead of $\delta$, Theorem \ref{thm:mr} also holds.

\bibliographystyle{alpha}
\bibliography{literatur}

\newcommand{\etalchar}[1]{$^{#1}$}
\begin{thebibliography}{BDHDH22}

\bibitem[ABCM01]{ABCM2001}
F.~Andreu, C.~Ballester, V.~Caselles, and J.~M. Maz{\'o}n.
\newblock Minimizing total variation flow.
\newblock {\em Differ. Integral Equ.}, 14(3):321--360, 2001.

\bibitem[BBH24a]{BBH24Arxiv}
Tim Binz, Felix Brandt, and Matthias Hieber.
\newblock Interaction of the primitive equations with sea ice dynamics, 2024.

\bibitem[BBH24b]{BBH24}
Tim Binz, Felix Brandt, and Matthias Hieber.
\newblock Rigorous analysis of the interaction problem of sea ice with a rigid
  body.
\newblock {\em Math. Ann.}, 389(1):591--625, 2024.

\bibitem[BCE25]{BCE2025}
Fran{\c{c}}ois Bouchut, Carsten Carstensen, and Alexandre Ern.
\newblock {{\(H^1\)}} regularity of the minimizers for the inviscid total
  variation and {Bingham} fluid problems for {{\(H^1\)}} data.
\newblock {\em Nonlinear Anal., Theory Methods Appl., Ser. A, Theory Methods},
  258:21, 2025.
\newblock Id/No 113809.

\bibitem[BDHDH22]{BDDH22}
Felix Brandt, Karoline Disser, Robert Haller-Dintelmann, and Matthias Hieber.
\newblock Rigorous analysis and dynamics of {Hibler}'s sea ice model.
\newblock {\em J. Nonlinear Sci.}, 32(4):26, 2022.
\newblock Id/No 50.

\bibitem[BH23]{BH23}
Felix Brandt and Matthias Hieber.
\newblock Time periodic solutions to {Hibler}'s sea ice model.
\newblock {\em Nonlinearity}, 36(6):3109--3124, 2023.

\bibitem[BLTT25]{BLTT25}
Daniel~W. Boutros, Xin Liu, Marita Thomas, and Edriss~S. Titi.
\newblock Global well-posedness of the elastic-viscous-plastic sea-ice model
  with the inviscid {Voigt}-regularisation.
\newblock Preprint, {arXiv}:2505.03080 [math.{AP}] (2025), 2025.

\bibitem[Bra25]{Brandt24}
Felix Christopher Helmut~Ludwig Brandt.
\newblock Well-posedness of {Hibler}'s parabolic-hyperbolic sea ice model.
\newblock Preprint, {arXiv}:2405.20198 [math.{AP}] (2025), 2025.

\bibitem[CKK23]{CK23}
Soufiane Chatta, Boualem Khouider, and M’hamed Kesri.
\newblock Linear well posedness of regularized equations of sea-ice dynamics.
\newblock {\em Journal of Mathematical Physics}, 64, 05 2023.

\bibitem[CMP{\etalchar{+}}74]{CMPRT1974}
M.~Coon, G.~Maykut, R.~Pritchard, D.~A. Rothrock, and A.~S. Thorndike.
\newblock Modeling the pack ice as an elastic-plastic material.
\newblock {\em AIDJEX Bulletin}, 24:1--105, 1974.

\bibitem[DL76]{DL76}
Georges Duvaut and Jacques-Louis Lions.
\newblock {\em Inequalities in Mechanics and Physics}, volume 219 of {\em
  Grundlehren der mathematischen Wissenschaften}.
\newblock Springer Berlin Heidelberg, 1976.

\bibitem[FP03]{FP2003}
Xiaobing Feng and Andreas Prohl.
\newblock Analysis of total variation flow and its finite element
  approximations.
\newblock {\em M2AN, Math. Model. Numer. Anal.}, 37(3):533--556, 2003.

\bibitem[GBC{\etalchar{+}}20]{BCEFGHHJPPSSW20}
Kenneth~M. Golden, Luke~G. Bennetts, Elena Cherkaev, Ian Eisenman, Daniel
  Feltham, Christopher Horvat, Elizabeth Hunke, Christopher Jones, Donald~K.
  Perovich, Pedro Ponte-Casta{\~n}eda, Courtenay Strong, Deborah Sulsky, and
  Andrew~J. Wells.
\newblock Modeling sea ice.
\newblock {\em Notices Am. Math. Soc.}, 67(10):1535--1555, 2020.

\bibitem[GK19]{GK2019}
Ugo Gianazza and Colin Klaus.
\newblock {{\(p\)}}-parabolic approximation of total variation flow solutions.
\newblock {\em Indiana Univ. Math. J.}, 68(5):1519--1550, 2019.

\bibitem[GLS13]{GLS13}
Oksana Guba, Jens Lorenz, and Deborah Sulsky.
\newblock On well-posedness of the viscous–plastic sea ice model.
\newblock {\em Journal of Physical Oceanography}, 43(10):2185 -- 2199, 2013.

\bibitem[Gra99]{Gray99}
J.~M. N.~T. Gray.
\newblock Loss of hyperbolicity and ill-posedness of the viscous–plastic sea
  ice rheology in uniaxial divergent flow.
\newblock {\em Journal of Physical Oceanography}, 29(11):2920 -- 2929, 1999.

\bibitem[HABea20]{HABB2020}
Elizabeth Hunke, Richard Allard, Philippe Blain, and et~al.
\newblock Should sea-ice modeling tools designed for climate research be used
  for short-term forecasting?
\newblock {\em Current Climate Change Reports}, 6:121--136, 2020.

\bibitem[HD97]{Hunke1997}
Eric~C. Hunke and John~K. Dukowicz.
\newblock An elastic–viscous–plastic model for sea ice dynamics.
\newblock {\em Journal of Physical Oceanography}, 27(9):1849--1867, 1997.

\bibitem[HD02]{DH02}
Elizabeth~C. Hunke and John~K. Dukowicz.
\newblock The elastic–viscous–plastic sea ice dynamics model in general
  orthogonal curvilinear coordinates on a sphere—incorporation of metric
  terms.
\newblock {\em Monthly Weather Review}, 130(7):1848 -- 1865, 2002.

\bibitem[Hib79]{Hib79}
W.D. Hibler.
\newblock A dynamic thermodynamic sea ice model.
\newblock {\em Journal of physical oceanography}, 9(4):815--846, 1979.

\bibitem[HTC{\etalchar{+}}19]{HTCF2019}
H.~D. B.~S. Heorton, M.~Tsamados, S.~T. Cole, A.~M.~G. Ferreira, A.~Berbellini,
  M.~Fox, and T.~W.~K. Armitage.
\newblock Retrieving sea ice drag coefficients and turning angles from in situ
  and satellite observations using an inverse modeling framework.
\newblock {\em Journal of Geophysical Research: Oceans}, 124(8):6388--6413,
  2019.

\bibitem[KDL15]{KDL2015}
Madlen Kimmritz, Sergey Danilov, and Martin Losch.
\newblock On the convergence of the modified elastic-viscous-plastic method for
  solving the sea ice momentum equation.
\newblock {\em J. Comput. Phys.}, 296:90--100, 2015.

\bibitem[KHLF00]{KHLF00}
Martin Kreyscher, Markus Harder, Peter Lemke, and Gregory~M. Flato.
\newblock Results of the sea ice model intercomparison project: Evaluation of
  sea ice rheology schemes for use in climate simulations.
\newblock {\em Journal of Geophysical Research: Oceans}, 105(C5):11299--11320,
  2000.

\bibitem[Liu11]{Liu11}
Wei Liu.
\newblock Existence and uniqueness of solutions to nonlinear evolution
  equations with locally monotone operators.
\newblock {\em Nonlinear Anal., Theory Methods Appl., Ser. A, Theory Methods},
  74(18):7543--7561, 2011.

\bibitem[LKT{\etalchar{+}}12]{LKTHL12}
J.-F. Lemieux, D.~A. Knoll, B.~Tremblay, D.~M. Holland, and M.~Losch.
\newblock A comparison of the jacobian-free newton-krylov method and the evp
  model for solving the sea ice momentum equation with a viscous-plastic
  formulation: A serial algorithm study.
\newblock {\em Journal of Computational Physics}, 231(17):5926--5944, 2012.

\bibitem[LMC{\etalchar{+}}10]{LMCHH10}
M.~Losch, D.~Menemenlis, J.-M. Campin, P.~Heimbach, and C.~Hill.
\newblock On the formulation of sea-ice models. part 1: Effects of different
  solver implementations and parameterizations.
\newblock {\em Ocean Modelling}, 33(1-2):129--144, 2010.

\bibitem[LTT22]{LTT22}
Xin Liu, Marita Thomas, and Edriss~S. Titi.
\newblock Well-posedness of {Hibler}'s dynamical sea-ice model.
\newblock {\em J. Nonlinear Sci.}, 32(4):31, 2022.
\newblock Id/No 49.

\bibitem[MDL{\etalchar{+}}21]{MDLLHRBHK2021}
C.~Mehlmann, S.~Danilov, M.~Losch, J.~F. Lemieux, N.~Hutter, T.~Richter,
  P.~Blain, E.~C. Hunke, and P.~Korn.
\newblock Simulating linear kinematic features in viscous-plastic sea ice
  models on quadrilateral and triangular grids with different variable
  staggering.
\newblock {\em Journal of Advances in Modeling Earth Systems},
  13(11):e2021MS002523, 2021.
\newblock e2021MS002523 2021MS002523.

\bibitem[MK21]{MK2021}
Carolin Mehlmann and Peter Korn.
\newblock Sea-ice dynamics on triangular grids.
\newblock {\em J. Comput. Phys.}, 428:18, 2021.
\newblock Id/No 110086.

\bibitem[SK18]{SK18}
C.~Seinen and B.~Khouider.
\newblock Improving the jacobian free newton-krylov method for the
  viscous-plastic sea ice momentum equation.
\newblock {\em Physica D: Nonlinear Phenomena}, 376:78--93, 2018.

\bibitem[Zei90]{Zei90B}
Eberhard Zeidler.
\newblock {\em Nonlinear functional analysis and its applications. {II}/{B}:
  {Nonlinear} monotone operators. {Transl}. from the {German} by the author and
  by {Leo} {F}. {Boron}}.
\newblock New York etc.: Springer-Verlag, 1990.

\bibitem[ZH97]{ZH97}
J.~Zhang and W.~Hibler.
\newblock On an efficient numerical method for modeling sea ice dynamics.
\newblock {\em Journal of Geophysical Research: Oceans}, 102(C4):8691--8702,
  1997.

\end{thebibliography}

\end{document}